\documentclass[12pt,a4paper]{article}
\usepackage{amsmath,amssymb,amsthm, todonotes,url,mathtools}
\usepackage{hyperref, xcolor}
\usepackage{bbm}
\usepackage[inline]{enumitem}
\usepackage[lmargin=35mm,rmargin=35mm,bmargin=30mm,tmargin=30mm]{geometry}
\usepackage{authblk}

\title{Reconstructing the degree sequence of a sparse graph from a partial deck}

\author{Carla Groenland\footnote{Utrecht University, Utrecht, The Netherlands, \url{c.e.groenland@uu.nl}.} \quad 
Tom Johnston\footnote{University of Oxford, Oxford, UK, \url{thomas.johnston@maths.ox.ac.uk}, \\\url{scott@maths.ox.ac.uk}, \url{jane.tan@maths.ox.ac.uk}. } \quad 
Andrey Kupavskii\footnote{G-SCOP, CNRS, University Grenoble-Alpes, France and Moscow Institute of Physics and Technology,
Russia, \url{kupavskii@yandex.ru}.} \quad
Kitty~Meeks\footnote{University of Glasgow, Glasgow, UK, \url{kitty.meeks@glasgow.ac.uk}.\\
C. Groenland is partially supported by the project CRACKNP that has received funding from the European Research Council under the EU Horizon 2020 research and innovation programme (grant agreement no. 853234).  A. Kupavskii is supported by the grant RSF 22-11-00131.
 K. Meeks is supported by a Personal Research Fellowship from the Royal Society of Edinburgh, funded by the Scottish Government.
A. Scott is supported by EPSRC grant EP/V007327/1.
} \quad
Alex Scott\protect\footnotemark[2] \quad
Jane Tan\protect\footnotemark[2]}
\date{\today}
\newtheorem{theorem}{Theorem}

\newtheorem{lemma}[theorem]{Lemma}
\newtheorem{corollary}[theorem]{Corollary}

\newtheorem{observation}[theorem]{Observation}

\newtheorem*{claim*}{Claim}



\makeatletter
\newtheorem*{rep@theorem}{\rep@title}
\newcommand{\newreptheorem}[2]{
\newenvironment{rep#1}[1]{
 \def\rep@title{#2 \ref{##1}}
 \begin{rep@theorem}}
 {\end{rep@theorem}}}
\makeatother

\newtheorem{conjecture}[theorem]{Conjecture}
\newtheorem*{conjecture*}{Conjecture}
\newreptheorem{conjecture}{Conjecture}

\theoremstyle{definition}                    

\theoremstyle{remark}   

\newtheorem*{remark*}{Remark}

\newcommand{\floor}[1]{\left\lfloor #1 \right\rfloor}

\newcommand{\N}{\mathbb{N}}

\usepackage{dsfont}
\newcommand{\eps}{\varepsilon}

\begin{document}
\maketitle

\begin{abstract}
The deck of a graph $G$ is the multiset of cards $\{G-v:v\in V(G)\}$. Myrvold (1992) showed that the degree sequence of a graph on $n\geq7$ vertices can be reconstructed from any deck missing one card.
We prove that the degree sequence of a graph with average degree $d$ can reconstructed from any deck missing $O(n/d^3)$ cards. In particular, in the case of graphs that can be embedded on a fixed surface (e.g. planar graphs), the degree sequence can be reconstructed even when a linear number of the cards are missing. 
\end{abstract}

\vspace{2mm}
\section{Introduction}

Throughout this paper, all graphs are finite and undirected with no loops or multiple edges. 
Given a graph $G$ and any vertex $v\in V(G)$, the \textit{card} $G-v$ is the subgraph of $G$ obtained by removing the vertex $v$ together with all edges incident to $v$. The \emph{deck} $\mathcal{D}(G)$ is the multiset of all unlabelled cards of $G$.

Kelly and Ulam \cite{K42,K57,U60} raised the natural question: is it possible for two non-isomorphic graphs to have the same deck?  
\begin{conjecture}[Reconstruction Conjecture]
For $n\geq 3$, two graphs $G$ and $H$ of order $n$ are isomorphic if and only if $\mathcal{D}(G)=\mathcal{D}(H)$.
\end{conjecture}

The Reconstruction Conjecture remains open, even for basic classes of graphs such as planar graphs and graphs of bounded maximum degree. There have been numerous surveys of partial results and related problems, for instance \cite{AFLM10, B91, BH77, N78}. For a more detailed introduction, we refer to \cite{LS16}.

One closely related problem is that of determining which graph parameters are \textit{reconstructible} in the sense that they are determined by the deck. For instance, given a full deck of cards, one can reconstruct the number of edges $m$ by summing over the number of edges present in all of the cards and dividing by $n-2$, where $n$ is the number of vertices. The degree sequence can then be easily deduced as well. Other reconstructible parameters include connectedness, planarity, the number of Hamiltonian cycles and the chromatic polynomial (all of these are discussed in \cite{LS16} together with further examples).

Some of these parameters can be reconstructed even when a subset of the deck is missing:
\begin{itemize}
    \item Bowler, Brown, Fenner and Myrvold~\cite{BBFM11} showed that any $\left\lfloor\frac n2\right\rfloor+2$ cards suffice to determine whether the graph is connected. 
    \item Groenland, Guggiari and Scott~\cite{GGS20} proved that the number of edges can be reconstructed from any $n-\frac1{20}\sqrt{n}$ cards. This improves on the work of Myrvold~\cite{M92} and the work of Brown and Fenner~\cite{BF18}, who proved the bounds $n-1$ and $n-2$ respectively.
    \item Myrvold~\cite{M88,M92} showed that the degree sequence is reconstructible from any $n-1$ cards.
\end{itemize}

The main contribution of the present paper is a proof that the degree sequence of graphs with bounded average degree can be reconstructed even with a linear number of missing cards.
\begin{theorem}
\label{thm:degree_partial}
Let $G$ have $n\geq 3$ vertices and average degree bounded above by some $d\in \N$. Then the degree sequence of $G$ can be reconstructed from its deck when at most $\frac{n}{10^4d^3}$ of the cards are missing. 
\end{theorem}
Since graphs that can be embedded on a fixed surface have bounded average degree, the following corollary is immediate.
\begin{corollary}
\label{cor:surfaces}
For any surface $S$, there is an $\varepsilon>0$ such that for any $n$-vertex graph $G$ embeddable on $S$, the degree sequence of $G$ can be reconstructed from any collection of at least $(1-\varepsilon)n$ cards.
\end{corollary}

In particular, our results hold for planar graphs. A graph class $\mathcal{C}$ is \emph{recognisable} if no graph $G\in \mathcal{C}$ has the same deck of cards as any graph $H\not \in \mathcal{C}$. Planar graphs are of particular interest as they were shown to be recognisable by Bilinksi, Kwan and Yu in ~\cite{bilinski2007reconstruction}, but are still not known to be reconstructible in general.  Partial results include the reconstructibility of outerplanar graphs~\cite{Giles74}, maximal planar graphs~\cite{Lauri81}, and certain 5-connected planar graphs~\cite{bilinski2007reconstruction}.

To reconstruct the degree sequence, we will first reconstruct the number of edges from the partial deck.

\begin{theorem}
\label{thm:edgecount}
Let $G$ be a graph on $n\geq 3$ vertices with average degree at most $d$ for some given $d \in \mathbb{N}$. Then the number of edges in $G$ can be reconstructed from any deck missing at most $\frac{n}{4d+6} - d - 5$ cards.
\end{theorem}
The proof of Theorem~\ref{thm:edgecount} is contained in Section~\ref{sec:cliquerec}, in which we also give a more general result for counting cliques in $G$. Section~\ref{sec:degreerecon} is devoted to proving Theorem~\ref{thm:degree_partial}. We discuss the tightness of our results and conclude with some open problems in Section~\ref{sec:concl}.

\subsection{Notation}
Throughout, we shall let $G$ be the original graph on $n$ vertices to be reconstructed. We write $G_i = G - v_i$ and order the vertices as $v_1,\ldots, v_n$ such that $G_1, \dots, G_{n-k}$ are the cards that have been given from the deck $\mathcal{D}(G)$. 

For any graph $H$, let $d_t(H)$ be the number of vertices of degree $t$ in $H$. That is, 
\[d_t(H) = |\{v\in V(H): d_H(v) = t\}|,\] 
where $d_H(v)$ denotes the degree of $v$ in $H$.  Similarly, let 
\[d_{< t}(H) = |\{v\in V(H): d_H(v) < t\}|\]
be the number of vertices of degree less than $t$ in $H$. 

We use the notation $[n]=\{1,\dots,n\}$, and write $[a,b]$ to denote the set of integers between $a$ and $b$ inclusive.

\section{Edge and clique counts}
\label{sec:cliquerec}

In order to reconstruct the degree sequence, we first reconstruct the number of edges. We note the following simple fact. 
\begin{observation}\label{obs:easydegrees}
A graph with average degree bounded above by $d\in \N$ has at least $\frac{n}{2}$ vertices of degree at most $2d$ and at least $\frac{n}{d+1}$ vertices of degree at most $d$.
\end{observation}
Given enough cards, this will allow us to assume that the card with the most edges corresponds to a low-degree vertex, and therefore has small ``error''. Another important feature is that the property of having small average degree is recognisable from the partial deck in the following sense.

\begin{lemma}
\label{lem:recognition_d}
Let $G$ be a graph on $n\geq 8$ vertices with average degree $d^*$. From any deck of $G$ missing $k\leq \frac{n}{4}$ cards, we can reconstruct a quantity $\widetilde{d}$ that satisfies $0\leq \widetilde{d}-d^*<1$.
\end{lemma}
\begin{proof}
  By \cite[Lemma 2.1]{GGS20}, we can calculate from the cards $G_1,\dots,G_{n-k}$ an estimate $\widetilde{m}$ for the number of edges $m$ that satisfies 
  \[
  0 \leq \widetilde{m}- m \leq \frac{k(n-1)}{n-2-k}.
  \]
  Since $k\leq \frac14n$, we find
  \[
  \frac{k(n-1)}{n-2-k}< \frac{n^2}{3n-8}\leq  \frac{n}2
  \]
  for $n\geq 8$. Hence $\widetilde{d}=\frac{2\widetilde{m}}n$ satisfies the claimed inequality.
\end{proof}

Lemma \ref{lem:recognition_d} allows us to assume that the average degree $d^*$ is known up to a potential error of 1. This approximation will be used to prove Theorem~\ref{thm:edgecount}, but once we have that result and enough cards for it to apply, the average degree can then be computed exactly. Thus, when we come to reconstructing the degree sequence we will be able to assume that $d^*$ is known.

\begin{proof}[Proof of Theorem~\ref{thm:edgecount}]
	Let $k$ be the number of missing cards, where $k\leq n/(4d+6)-d-5$ and let the partial deck of cards consist of $G_1, \dots, G_{n-k}$ with $G_i = G - v_i$.
	After possibly reordering the cards, we may assume that $|E(G_1)| \geq |E(G_2)| \geq \dotsb \geq |E(G_{n-k})|$ or, equivalently, $d_G(v_1) \leq \dotsb \leq d_G(v_{n-k})$.
	We may assume that $n\geq 2(d + 6)(2d+3) \geq 36$, else there would be no cards missing. So by Lemma~\ref{lem:recognition_d}, it follows that any graph $H$ with this partial deck has average degree at most $d+1$. Then Observation~\ref{obs:easydegrees} and the conditions on $k$ together imply that any such graph has at least $k+1$ vertices of degree at most $d+1$ and in particular the vertex $v_1$ has degree at most $d+1$. Therefore, the corresponding card $G_1$ must satisfy
	\begin{equation}
	\label{eq:dbounds}
	d_{<t}(G_1) \in [d_{<t}(G)-1, d_{<t}(G)+d+1]
	\end{equation}
	for each $t\in[0,n-1]$. The upper bound comes from the observation that $v_1$ has at most $d+1$ neighbours. Only these vertices can drop degree from $t$ to $t-1$ when $v_1$ is deleted, and hence be counted in $d_{<t}(G_1)$ but not in $d_{<t}(G)$. For the lower bound, since the degree of a vertex in $G_1$ is at most the degree of the corresponding vertex in $G$, the only possible loss comes from the possibility that the deleted vertex $v_1$ may have had degree less than $t$. Applying Observation~\ref{obs:easydegrees} first and then (\ref{eq:dbounds}), we find
\[\frac12{n} \leq \sum_{t=0}^{2(d+1)}d_{t}(G) \leq 1 + \sum_{t=0}^{2(d+1)} d_{t}(G_1).\]
	It follows that there must be some $t\in [0,2(d+1)]$ such that 
	\begin{equation}
	\label{eqn:g1bounds}
	d_t(G_1) \geq  \frac{1}{2d + 3}\left( \frac12 n - 1 \right) \geq k + d +4,
	\end{equation}
	where the last inequality holds by our assumptions on $k$. Let us choose $t$ to be the smallest integer satisfying (\ref{eqn:g1bounds}), noting that this is determined by $G_1$ and does not depend on any other information about $G$. Our next goal is to find a card corresponding to a vertex with degree exactly $t$. 
	
	Set $j := d_{< t}(G_1)$. We claim that $d_G(v_{j+2}) = t$. From (\ref{eq:dbounds}), we see that $d_G(v_{j + 2}) \geq t$ since $j + 2=d_{<t}(G_1)+2 > d_{< t}(G)$. Moreover, 
	\begin{align*}
	j + 2 &= d_{<t+1}(G_1)-d_t(G_1) + 2 \\
	&\leq d_{< t+1}(G)+d+1-(k+d+4)+2 <  d_{<t+1}(G)- k.
	\end{align*}
	by the bounds in (\ref{eq:dbounds}) and (\ref{eqn:g1bounds}). Since we are missing at most $k$ cards from our deck, the $(j+2)$nd card is certainly within the first $j+k+2$ cards in the whole deck. Hence, we also have the reverse inequality $d_G(v_{j+2}) \leq t$. This proves our claim that $d_G(v_{j+2}) = t$. Since we can compute $j$ and $t$, and we have the card $G_{j+2}$, the number of edges may now be reconstructed by the formula $|E(G)| = |E(G_{j+2})| + t$.
\end{proof}

The preceding proof extends easily to reconstructing clique counts from an incomplete deck by replacing $d_t(H)$ and $d_{<t}(H)$ with analogous notions in terms of ``clique degree". Namely, for each fixed $r \in \mathbb{N}$, let $c(v)$ be the number of $r$-cliques which contain the vertex $v$. Then let the number of vertices $v$ for which $c(v) = t$ (similarly, $c(v) < t$) be denoted by $c_t(H)$ ($c_{<t}(H)$). Any vertex of degree at most $2(d+1)$ is in at most $\binom{2(d+1)}{r-1}$ cliques, and therefore
\[ \frac{n}{2} \leq 1 + \sum_{t=0}^{2(d+1)} d_{t}(G_1) \leq 1 + \sum_{t=0}^{\binom{2(d+1)}{r-1}} c_{t}(G_1).
\]
By the pigeonhole principle, there is some $t\in \{0,\dots,2(d+1)\}$ such that 
\begin{equation*}
c_t(G_1) \geq  \frac{1}{\binom{2(d+1)}{r-1} + 1}\left( \frac{n}{2} - 1 \right).
\end{equation*}
Again, only the neighbours of $v_1$ may appear in fewer cliques in $G_1$ than in $G$, and so $c_{<t}(G_1) \in [c_{<t}(G)-1, c_{<t}(G)+d_G(v_1)]$. 
Since the number of $r$-cliques in $G_{j+2}$ is the number of $r$-cliques in $G$ minus $c(v_{j+2})$, it suffices to choose $k$ to guarantee $c(v_{j+2}) =t$ for $j = c_{< t}(G_1)$. We obtain the following result.
\begin{theorem}\label{thm:cliquecount}
Let $d,r\in \N$. For any graph $G$ on $n$ vertices with average degree at most $d$, the number of cliques of size $r$ in $G$ can be reconstructed from any deck missing at most $\left(1 + \binom{2(d+1)}{r-1}\right)^{-1}\left(\frac{n}{2} -1 \right) - d - 5$ cards.
\end{theorem}

\section{Degree sequence reconstruction}
\label{sec:degreerecon}
Once we know the number of edges $m$ in a graph $G$, deducing its degree sequence from the complete deck is a simple matter of subtracting from $m$ the number of edges seen in each card. Losing one card $G_i$ does not pose a problem as the missing degree is given by $d_G(v_i) = 2m - \sum_{j\neq i} d_G(v_j)$. However, as soon as we are missing just two cards it is no longer known whether the degree sequence can still be reconstructed.

The main result in this section shows that for every graph $G$ on $n\geq 3$ vertices with average degree at most $d$ for some $d\in \N$, the degree sequence of $G$ can be reconstructed from its deck when at most $\frac{n}{10^4d^3}$ of the cards are missing. 
Since the degree sequence can be reconstructed if no cards are missing, we may assume that $n\geq 10^4d^3$, which implies that the number of missing cards is at most $\frac{n}{4d+6} - d - 5$. Applying Theorem~\ref{thm:edgecount} then allows us to reconstruct the number of edges in $G$. This means that we can determine $d_G(v_i)$ for all vertices corresponding to cards in our partial deck, as well as the average degree of $G$. 

The total number of occurrences of degree $t$ vertices across all of the cards is $\sum_{i=1}^n d_t(G_i)$. At the same time, each vertex $v$ of degree $t$ in $G$ still has degree $t$ in $n-(t+1)$ cards, namely all those $G_i$ for which $v_i \not\in N_G(v) \cup \{v\}$. A vertex $v$ of degree $t+1$ has degree $t$ in $G_i$ if and only if $v_i\in N_G(v)$. Hence  
\begin{align}
\sum_{i=1}^n d_t(G_i)=(n-1-t)d_t(G)+(t+1)d_{t+1}(G). \label{eq:deg_t_in_deck}
\end{align}
In order to guess $d_t(G)$ from $d_{t+1}(G)$ or vice versa, we first obtain a good estimate on the left-hand side of the equation above.
\begin{lemma}\label{lem:estimate_deg_t}
Suppose we know the number of edges of $G$ and that the average degree of $G$ is at most $d\in \N$. Moreover, assume that $k \leq \frac{n}{1100d^2}$ is the number of missing cards and $n\geq 10^4 d^3$ is the number of vertices. Then for every $t\in [0,n]$, we can reconstruct an estimate $\widetilde{s_t}$ from the given cards such that $|\widetilde{s_t}-\sum_{i=1}^n d_t(G_i)|<\frac{n}8$.
\end{lemma}
\begin{proof}
Fix any $t\in[0,n]$. We again label the given cards $G_1,\dots,G_{n-k}$ so that $|E(G_1)|\geq \dots \geq |E(G_{n-k})|$. Let the missing cards be $G_{n-k+1}, \ldots, G_{n}$ ordered arbitrarily. 

To estimate $\sum_{i=1}^n d_t(G_i)$, we partition $[n]$ into three sets $I_1$, $I_2$, and $I_3$ defined as follows:
\begin{align*}
I_1&= \{i \in [2,n]: d_{G_1}(v_i)>100 d^2\},\\
I_2&=\{i\in[n-k]: d_{G}(v_i)\le 100 d^2\},\\
I_3&=[n] - (I_1\cup I_2).
\end{align*}
We assume that the vertex numbering of $G_1$ is inherited from $G$ for the sake of the argument, but we do not exploit that these labels are present on our given card. In particular, we do not access the set $I_1$, only the multiset $\{d_{G_1}(v_i):i\in [2,n] \text{ with } d_{G_1}(v_i)>100d^2\}$.
Note that $I_1\cap I_2=\emptyset$ as $d_{G_1}(w)\leq d_G(w)$ for all $w\in V(G_1)$, so we can write $[n] = I_1 \sqcup I_2 \sqcup I_3$ as a disjoint union. Moreover, note that $1\in I_2$ by Observation \ref{obs:easydegrees}.

For each $j=1,2,3$, we estimate $\sum_{i\in I_j} d_t(G_i)$. Recall that we know the number of edges of $G$ and hence the degrees of $v_1,\dots,v_{n-k}$. This is enough to reconstruct the set $I_2$, and to read off $d_t(G_i)$ for each $v_i\in I_2$ by examining the relevant card. Therefore, we can determine $\sum_{i\in I_2} d_t(G_i)$ exactly.

We estimate $\sum_{i\in I_1} d_t(G_i)$ by $\sum_{i\in I_1}d_t(G_1-v_i)$, and we now bound the error of this estimation. The vertex $v_1$ has degree at most $d$ in $G$ by Observation \ref{obs:easydegrees}, and so, for each $i\in [2,n]$, the vertex $v_1$ has degree at most $d$ in the graph $G_i$. It follows that for all $i\in [2,n]$, 
\[
|d_t(G_i)-d_t(G_1-v_i)|=|d_t(G_i)-d_t(G_i-v_1)|\le d + 1.
\]
Since there are at most $\frac{n}{100d}$ vertices with degree greater than $100d^2$ in $G$ and hence also at most $\frac{n}{100d}$ such vertices in $G_1$, we find that
\begin{equation}
\sum_{i\in I_1} |d_t(G_i)-d_t(G_1-v_i)|\le (d+1)\cdot |I_1|\le (d+1)\cdot \frac n{100d} \leq\frac{2n}{100}. \label{eq:v1est}
\end{equation}

Finally, we can express $I_3$ as the union
\[
\{i\geq  n-k+1: d_{G}(v_i)\le 100 d^2\}\cup\{i>1: d_{G_1}(v_i)\le 100 d^2 \text{ and }d_G(v_i)>100 d^2\}.
\]
In this form, we see that all vertices $v_i$ with $i\in I_3$ have degree at most $100 d^2+1$ in $G$. The first set in the union has cardinality at most $k$ and the second has cardinality at most $d_G(v_1)\le d$ (since all such vertices must be adjacent to $v_1$). Thus $|I_3|\leq k+d$. Moreover, observe that $|d_t(G_i)-d_t(G_j)|\le d_G(v_i)+d_G(v_j)+1$. This implies that 
\begin{equation}
\sum_{i \in I_3}|d_t(G_i)-d_t(G_1)|\le (100d^2+d+2)|I_3|\le (100d^2+d+2)(k+d) \label{eq:v3est}
\end{equation}
so we can estimate $\sum_{i\in I_3} d_t(G_i)$ by $|I_3|d_t(G_1)$. Note that we can reconstruct $|I_3|=n-|I_1|-|I_2|$ from the cards.

We now estimate $\sum_{i=1}^n d_t(G_i)$ by
\[
\widetilde{s_t}= \sum_{i\in I_1}d_t(G_1-v_i)+\sum_{i\in I_2}d_t(G_i)+ |I_3|d_t(G_1), 
\]
which is reconstructible from our partial deck.

Using \eqref{eq:v1est} and \eqref{eq:v3est}, the margin of error $\left| \sum_{i=1}^n d_t(G_i) - \widetilde{s_t} \right|$ is then given by
\begin{align*}
&\quad\left| \sum_{i\in I_1} d_t(G_i) +\sum_{i\in I_3} d_t(G_i)- \sum_{i\in I_1}d_t(G_1-v_i)  - |I_3|d_t(G_1)  \right|\\
&\leq \sum_{i\in I_1} \left|d_t(G_i)- d_t(G_1-v_i) \right| + \sum_{i\in I_3}  \left| d_t(G_i) - d_t(G_1) \right|\\
&\leq \frac{2n}{100} + (100d^2+d+2)(k+d)
\end{align*}
and this is less than $\frac{n}{8}$ for $k \leq \frac n{1100d^2}$ and $n\geq 10^4 d^3$.
\end{proof}

We now deduce the proof of the main result.
\begin{proof}[Proof of Theorem \ref{thm:degree_partial}]
Following the discussion at the start of this section, we may assume that $n\geq 10^4d^3$, that we have already reconstructed the number of edges in $G$, and that we have therefore determined the best possible upper bound $d\in \N$ on the average degree. In particular, for every $t\in [0,n]$, Lemma~\ref{lem:estimate_deg_t} provides an estimate $\widetilde{s_t}$ for $\sum_{i=1}^n d_t(G_i)$ with $|\widetilde{s_t}-\sum_{i=1}^n d_t(G_i)|<n/8$.

Rewriting (\ref{eq:deg_t_in_deck}), we obtain 
\[
d_t(G)=\frac1{(n-1-t)}\left(\sum_{i=1}^n d_t(G_i)- (t+1)d_{t+1}(G)\right)
\]
and, estimating $\sum_{i=1}^n d_t(G_i)$ by $\widetilde{s_t}$, we obtain the following estimate for $d_t(G)$
\[\widetilde{d_t} = \frac1{(n-1-t)}\left(\widetilde{s_t}-(t+1)d_{t+1}(G)\right). \]
If $t+1\leq \frac{3n}4$, then $\frac{n}8\leq \frac12(n-1-t)$ and hence 
\begin{equation}\label{eq:roundable}
\left|d_t(G)- \widetilde{d_t} \right| = \frac1{(n-1-t)} \left| \sum_{i=1}^n d_t(G_i) - \widetilde{s_t} \right| <\frac12.
\end{equation}
If $d_{t+1}$ is known exactly, this means that we can reconstruct $d_t$ exactly by rounding $\frac1{(n-1-t)}\left(\widetilde{s_t}-(t+1)d_{t+1}(G)\right)$ to the nearest integer. A symmetric argument, obtained by solving (\ref{eq:deg_t_in_deck}) for $d_{t+1}(G)$ and using the same estimate $\widetilde{s_t}$, shows that we can also reconstruct $d_{t+1}$ given $d_t$ and the partial deck when $t\geq \frac14n$.

We now show that there is a $t \in \left[\frac14n,\frac34n\right]$ for which we can reconstruct that $d_t(G)=0$. Observe that if there are two cards in our partial deck with no vertices of degree $t$ or $t-1$, then $d_t(G) = 0$. Moreover, if there is a $t$ such that $d_{t-1}(G)$, $d_{t}(G)$ and  $d_{t+1}(G)$ are all 0, then no vertices of degree $t - 1$ or $t$ will appear on any card, and we can reconstruct that $d_t(G) = 0$. We claim that this is the case for some $t \in \left[\frac14n,\frac34n\right]$. Suppose for a contradiction that $d_{t-1}(G)+d_t(G)+d_{t+1}(G) \geq 1$ for all $t \in \left[\frac14n,\frac34n\right]$. 
Then 
\[
\sum_{t\in \left[\frac14n,\frac34n\right]}d_t(G) \geq \frac13 \sum_{t\in \left[\frac14n+1,\frac34n-1\right]}(d_{t-1}(G)+d_t(G)+d_{t+1}(G))\geq \frac13\cdot 1\cdot \left(\frac{n}2-4\right)
\] 
which implies that 
\[dn\geq \sum_{t\in \left[\frac14n,\frac34n\right]}t d_t(G) \geq \frac1{4\cdot 6}n(n-8).
\]
This contradicts the assumption that $n \geq 10^4 d^3$.

Fix a $t \in \left[\frac14n,\frac34n\right]$ such that there no vertices of degree $t$ or $t-1$ on any card as found above. Since $d_t(G)$ is known exactly (to be 0), we may now reconstruct the estimate for $d_{t-1}$ given in \eqref{eq:roundable} and round to determine $d_{t-1}$ exactly. This process allows us to iteratively reconstruct $d_{t-1}(G),\dots,d_0(G)$. Returning to $d_t(G)$, we can also `push' in the other direction using the symmetric estimate to determine $d_{t+1}(G),d_{t+2}(G),\dots,d_{n-1}(G)$ in order as well.
\end{proof}
\section{Conclusion}
\label{sec:concl}
We have shown that it is possible to reconstruct the degree sequence of planar graphs with a linear number of missing cards. This is tight up to a constant. For example, consider the graphs  
\[
G_1=K_{1,p+1}\sqcup K_{1,p+1}\sqcup K_{1,p-1} \text{ and }
G_2=K_{1,p+1}\sqcup K_{1,p}\sqcup K_{1,p}
\]
formed by the disjoint union of three stars. For both graphs, roughly two thirds of their cards are equal to $K_{1,p+1}\sqcup K_{1,p}\sqcup K_{1,p-1}$, and we might be unable to distinguish the two graphs even with nearly two thirds of the deck. Yet $G_1$ has two vertices of degree $p+1$, whereas $G_2$ has only one such vertex. These graphs do have the same number of edges, but we can find examples with a linear number of common cards and a different number of edges. For example, $K_{2,p} \sqcup K_{1,p}$ and $K_{2,p+1} \sqcup K_{1, p-1}$ share approximately half their cards yet have a different number of edges.

These examples can be generalised to graph classes with a larger (constant) average degree $d$ as well. 
Indeed, consider adding disjoint copies of the same $(3p+4)$-vertex graph $H$ to both $G_1$ and $G_2$. The resulting graphs will still have about $\frac23\times \frac12=\frac13$ of their cards in common, and we can create the desired average degree by choosing the density of $H$.

Let $cc(G,H)$ denote the number of cards that $G$ and $H$ have in common, and let $cc(n):=\max\{cc(G,H) : G,H \text{ distinct graphs on $n$ vertices}\}$.
The graph reconstruction conjecture states that $cc(n) \leq n -1$ for $n\geq 3$. The examples above are variations on constructions by Bowler, Brown and Fenner \cite{BBF10} which lead to a bound $cc(n) \geq (\frac23 + o(1))n$. The authors of~\cite{BBF10} conjecture that the bound is tight and also propose a characterisation of the extremal graphs. 
\begin{conjecture}[Bowler, Brown and Fenner \cite{BBF10}]
\label{conj:BBF}
For large enough $n$, every graph is determined, up to isomorphism, by any $2\floor{(n-1)/3}+1$ of its vertex-deleted subgraphs.
\end{conjecture}
A good first step towards Conjecture \ref{conj:BBF} would be to determine whether $cc(n)\ge
(1-o(1))n$. A positive answer would disprove Conjecture \ref{conj:BBF}, whereas a negative answer would prove the Reconstruction Conjecture in a strong form. We remark that the answer to the equivalent question in the `small' cards set-up has been answered. 
Let $s(G)$ denote the smallest $\ell$ for which the multiset $\mathcal{D}_\ell(G)$ of $\ell$-vertex induced subgraphs of $G$ determines $G$, and let $s(n)=\max\{s(G):G \text{ graph on }n\text{ vertices}\}.$
N\'{y}dl \cite{Nydl92} proved that $s(n)\geq (1-o(1))n$ by constructing, for any $\eps>0$, two non-isomorphic graphs on $n$ vertices with the same set of $\ell$-vertex subgraphs for all $\ell\leq (1-\eps)n$.

Groenland, Guggiari and Scott \cite{GGS20} conjectured that the degree sequence of a graph can be reconstructed from a deck of cards with a constant number $k$ of missing cards (for $n$ sufficiently large). It follows from Theorem \ref{thm:degree_partial} that the conjecture holds for graphs where the average degree is at most $c_kn^{\frac{1}{3}}$ (for some $c_k$ depending only on $k$), but it is not yet known to hold for general graphs and we repeat it below.
\begin{conjecture}[Groenland, Guggiari and Scott \cite{GGS20}]
Fix $k\in\mathbb{N}$ and let $n$ be sufficiently large. For any graph $G$ on $n$ vertices, the degree sequence of $G$ is reconstructible from any $n-k$ cards.
\end{conjecture}

\paragraph{Acknowledgements}
AS would like to thank Xingxing Yu for helpful discussions that led to us thinking about this project. We would like to thank the referees for their helpful comments.

\bibliographystyle{scott}
\bibliography{reconstruction}

\begin{thebibliography}{10}

\bibitem{AFLM10}
K.~J. Asciak, M.~A. Francalanza, J.~Lauri and W.~Myrvold.
\newblock A survey of some open questions in reconstruction numbers.
\newblock {\em Ars Combin.}, {\bfseries 97}:443--456, 2010.

\bibitem{bilinski2007reconstruction}
M.~Bilinski, Y.~S. Kwon and X.~Yu.
\newblock On the reconstruction of planar graphs.
\newblock {\em Journal of Combinatorial Theory, Series B}, {\bfseries
  97}(5):745--756, 2007.

\bibitem{B91}
J.~A. Bondy.
\newblock A graph reconstructor’s manual.
\newblock {\em Surveys in Combinatorics}, {\bfseries \textbf{166}}:221--252,
  1991.

\bibitem{BH77}
J.~A. Bondy and R.~L. Hemminger.
\newblock Graph reconstruction—a survey.
\newblock {\em J. Graph Theory}, {\bfseries \textbf{1}}(3):227--268, 1977.

\bibitem{BBF10}
A.~Bowler, P.~Brown and T.~Fenner.
\newblock Families of pairs of graphs with a large number of common cards.
\newblock {\em J. Graph Theory}, {\bfseries \textbf{63}}(2):146--163, 2010.

\bibitem{BBFM11}
A.~Bowler, P.~Brown, T.~Fenner and W.~Myrvold.
\newblock Recognizing connectedness from vertex-deleted subgraphs.
\newblock {\em J. Graph Theory}, {\bfseries \textbf{67}}(4):285--299, 2011.

\bibitem{BF18}
P.~Brown and T.~Fenner.
\newblock The size of a graph is reconstructible from any $n- 2$ cards.
\newblock {\em Discrete Math.}, {\bfseries \textbf{341}}(1):165--174, 2018.

\bibitem{Giles74}
W.~B. Giles.
\newblock The reconstruction of outerplanar graphs.
\newblock {\em J. Combin. Theory Ser. B}, {\bfseries \textbf{16}}(3):215--226,
  1974.

\bibitem{GGS20}
C.~Groenland, H.~Guggiari and A.~Scott.
\newblock Size reconstructibility of graphs.
\newblock {\em J. Graph Theory}, {\bfseries 96}(2):326--337, 2021.

\bibitem{K42}
P.~J. Kelly.
\newblock {\em On isometric transformations}.
\newblock PhD thesis, University of Wisconsin, 1942.

\bibitem{K57}
P.~J. Kelly.
\newblock A congruence theorem for trees.
\newblock {\em Pacific J. Math.}, {\bfseries \textbf{7}}(1):961--968, 1957.

\bibitem{Lauri81}
J.~Lauri.
\newblock The reconstruction of maximal planar graphs.
\newblock {\em J. Combin. Theory Ser. B}, {\bfseries \textbf{30}}(2):196--214,
  1981.

\bibitem{LS16}
J.~Lauri and R.~Scapellato.
\newblock {\em Topics in graph automorphisms and reconstruction}.
\newblock London Mathematical Society Lecture Note Series {\bf 432}. Cambridge
  University Press, 2nd edition, 2016.

\bibitem{M88}
W.~Myrvold.
\newblock {\em Ally and adversary reconstruction problems}.
\newblock PhD thesis, University of Waterloo, 1988.

\bibitem{M92}
W.~Myrvold.
\newblock The degree sequence is reconstructible from $n- 1$ cards.
\newblock {\em Discrete Math.}, {\bfseries \textbf{102}}(2):187--196, 1992.

\bibitem{Nydl92}
V.~N{\'y}dl.
\newblock Finite undirected graphs which are not reconstructible from their
  large cardinality subgraphs.
\newblock {\em Discrete Math.}, {\bfseries \textbf{108}}(1-3):373--377, 1992.

\bibitem{N78}
C.~{St}. J. A. Nash-Williams.
\newblock The reconstruction problem.
\newblock {\em {Selected Topics in Graph Theory}}, {\bfseries
  \textbf{1}}:205--236, 1978.

\bibitem{U60}
S.~M. Ulam.
\newblock {\em {A Collection of Mathematical Problems}}, volume~8.
\newblock Interscience Publishers, 1960.

\end{thebibliography}

\end{document}